\documentclass{cedram-msia}

\usepackage{amsmath,amssymb,mathrsfs}
\usepackage{amsthm}
\usepackage{dsfont}
\usepackage{graphicx}
\usepackage{xcolor}
\usepackage{enumitem}
\usepackage{hyperref}
\usepackage[english]{babel}
\usepackage[T1]{fontenc}
\usepackage[utf8]{inputenc}

\newcommand{\p}{\partial}
\newcommand{\beq}{\begin{equation}}
\newcommand{\eeq}{\end{equation}}
\renewcommand{\epsilon}{\varepsilon}
\renewcommand{\leq}{\leqslant}
\renewcommand{\geq}{\geqslant}
\renewcommand{\d}{\mathrm{d} }
\newcommand{\R}{\mathbb{R}}
\newcommand{\N}{\mathbb{N}}
\newcommand{\Z}{\mathbb{Z}}
\newcommand{\Leb}{\mathrm{L}}
\newcommand{\ppt}{\frac{\p}{\p t}}
\newcommand{\supp}{\mathrm{supp \,}}

\newcommand{\M}{\mathcal{M}}
\newcommand{\e}{\mathrm e}
\newcommand{\ddt}{\frac{\mathrm{d}}{\mathrm{d}t}}

\newcommand{\1}{\mathds 1}
\newcommand{\Cr}{\mathcal{C}}

\title[Measure solutions for the selection equation]{Measure framework for the pure selection equation: global well posedness and numerical investigations}

\author{\firstname{Hugo} \lastname{Martin}}
\address{IRMAR\\
Universit\'e de Rennes\\
Campus de Beaulieu, bâtiments 22 et 23
263 avenue du Général Leclerc, CS 74205
35042 RENNES Cedex
France}
\email{hugo.martin@univ-rennes.fr}
  

\begin{document}

% Abstract.
\begin{abstract}
We study the classic pure selection integrodifferential equation, stemming from adaptative dynamics, in a measure framework by mean of duality approach. After providing a well posedness result under fairly general assumptions, we focus on the asymptotic behaviour of various cases, illustrated by some numerical simulations.
\end{abstract}

% Use the \maketitle command after the abstract
\maketitle

\section{Introduction}

\subsection{A brief state of the art}

The synthetic theory of evolution is nowadays the golden standard to explain the variety of living species, as well as their disappearing. This theory predicts fluctuations in populations due to the occurence of mutations, that are then eventually selected by the environment. For decades now, mathematical models have aimed at studying these phenomena, both from an individual point of view and at the population level.
\
This first approach leads to individual based models, that use stochastic processes~\cite{Dieckmann1996, Champagnat2002, Champagnat2006, Champagnat2008, Champagnat2010a, Champagnat2014, Kraut2019, Coquille2020}. It also provides a derivation of the deterministic counterpart of the studied model. When a large enough population is considered, the behaviours of individuals are averaged, resulting in deterministic equations, such as ordinary differential, integro-differential or partial differential equations. The famous selection-mutation model and its variants have been extensively studied for decades now~\cite{Aafif1998,Burger2000,Ferriere2002, Calsina2006, PerthameTransport,Cleveland2009, Canizo2012, Jabin2012, Ackleh2014, Ackleh2016,  Jabin2016, Bonnefon2017, Jabin2017, Costa2021, Ackleh2021}.
In this type of models, populations are structured by trait, and the competitive interactions often lead to the selection of the best trait. Various aditionnal features can be added, be it for example clonal selection and two populations interacting~\cite{Busse2016,Busse2020}, intra species cooperation or competition~\cite{Cooney2021}, horizontal gene transfer~\cite{Calvez2020a} or the addition of a space variable to account for the behaviour of tumor cells, see~\cite{Ardaseva2019, Lorenzi2019, Villa2020} and reference therein.
\\
This article focuses on the pure selection equation, that have already been studied by many authors, see~\cite{Ackleh1999,Ackleh2005,PerthameTransport,Desvillettes2008,Jabin2010} among others. One of its goals is to provide a measure framework to study this equation with a very general \textit{selection pressure operator}, denoted $\Sigma$. The resulting equation is then
\beq \label{eq:selection}
\left\{
\begin{array}{l}
\ppt n(t,x) = \Sigma[n(t,\cdot)](x)n(t,x) \quad t>0, x\in X\medskip\\
n(0,x) = n_0(x), \quad x\in X
\end{array}
\right.
\eeq
in which traits $x$ lie in a subset of $\R^d$ denoted $X$, with $d$ a positive integer. In past years, efforts have been made to state the considered models in spaces of measures~\cite{Ackleh1999,Ackleh2005,Cressman2005,Canizo2012}. The interest of such a general formulation, both of the type of solution and of the selection pressure operator, lies in the broad class of models included. Indeed, in the case of $X$ being a finite subset of $\R$, say $X = \{x_1,\dots,x_K\}$, Equation~\eqref{eq:selection} reduces to a system of ODE, such as competitive models in the sense of Hirsch~\cite{Hirsch1982, Hirsch1985, Hirsch1988}. For example, denoting $n(t,x_i)$ the amount of the $i$th population at time $t$, with $r_i$ its intrinsic growth rate and $\alpha_{i,j}$ the competition coefficients with the $j$th species, we recover the general system mentionned in~\cite{May1975} by choosing
\[
\Sigma[n(t,\cdot)](x_i) := r_i\left(1 - \sum_{j=1}^K \alpha_{i,j}n(t,x_j)\right),
\]
and even more general models are possible, see~\cite{Champagnat2010}.

\subsection{Framework for measure solutions by duality}

As previously mentionned, measure solutions to structured population equation have attracted much attention during the past years. Recent articles summon the theory of semigroups to express these kind of solutions, both in case of linear or nonlinear equation. We refer to~\cite{Duell2021} for a complete exposition of the relevant measure theory.
\\
We denote the set of bounded Borel functions defined on $X$ by $\mathcal{B}(X)$. For any $f$ lying in this set, the supremum norm is defined by \[\|f\|_\infty = \sup_{x\in X} |f(x)|.\]% and for any $\mu$ belonging to $\M(X)$ we use the notation \[\mu f := \int_X f\d \mu.\]
Endowed with the supremum norm, the space $\Cr(X)$ of continuous functions on $X$ is a Banach space. We can identify its topological dual space with $\M(X)$ the space of signed measures on $X$, thanks to the Riesz representation theorem, through the mapping
\begin{equation*}
\left\{
\begin{array}{lcr}
\M(X) &\rightarrow& (\Cr(X))' \medskip\\
\mu &\mapsto& (f\mapsto \langle\mu, f\rangle)
\end{array}
\right.
\end{equation*}
which is an isometric isomorphism, thus
\[\|\mu\|_{TV} = \sup_{\|f\|_\infty \leq 1} \langle\mu, f\rangle\]
defines a norm on $\M(X)$, so it is a Banach space. The standard Hahn-Jordan decomposition of a signed measure $\mu$ is
$\mu = \mu_+ - \mu_-$ with $\mu_+,\mu_-\in\M_+(X)$ the set of finite non negative measures, and these measures are mutually singular. This decomposition enables to define
\[
|\mu|:=\mu_+ + \mu_-
\]
which is a non negative measure. In turn, we can define the total variation norm of a measure $\mu$ by
\[\|\mu\|_{TV} := |\mu|(X) = \mu_+(X) + \mu_-(X).\]
It is worth stressing out that for a non negative measure $\mu$, one has
\[
\|\mu\|_{TV} = \mu(X) = \langle \mu,\1_X\rangle.
\]
Finally, we explain how to extend the classical sense of Equation~\eqref{eq:selection} to measures. Assume that $n(t,x)\in \Cr([0,\infty);\Leb^1(X))$ is differentiable in time and satisfies~\eqref{eq:selection} in the classical sense. Then, multiplying~\eqref{eq:selection} by $f\in\mathcal{B}(X)$ and integrating it in space and then time, we obtain
\[
\int_X f(x)n(t,x)\d x = \int_X f(x)n_0(x)\d x + \int_0^t \int_X \Sigma[n(s,\cdot)](x)f(x)n(s,x)\d x\,\d s.
\]
This equation leads us to the following definition of measure solution for Equation~\eqref{eq:selection}.
\begin{defi}\label{def:measure sol}
Let $T>0$. A family $(\mu_t)_{0\leq t \leq T}\in \Cr([0,T];\M(X))$ with initial measure $\mu_0$ is called a measure solution to the pure selection equation with initial data $\mu_0$ if for all $f\in \Cr_b(X)$ the mapping $t \mapsto \langle\mu_t,f\rangle$ is continuous on $[0,T]$, and for all $t\geq 0$ and all bounded measurable functions $f$ on $X$ one has
\begin{equation}\label{eq:measure}
\mu_tf=\mu_0 f + \int_0^t \mu_s\left(\Sigma[\mu_s]f\right)\d s
\end{equation}
% \int_X \e^{\int_0^t \Sigma[\mu_s](x)\d s}f(x) \d \muin(x)
\end{defi}

\subsection{Assumptions and well posedness result}

We make the following assumptions on the selection pressure operator. First, we require that for each measure $\mu$ bounded in total variation norm, the associated selection pressure $\Sigma[\mu]$ is a bounded Borel function
\[
\Sigma :\left\{
 \begin{array}{l}
\M_+(X) \to \mathcal{B}(X) \medskip\\
\mu \mapsto \Sigma[\mu].
\end{array}
\right.
\]
In addition, we require
\begin{equation} \label{hyp:lips}
\forall r>0,\, \exists k(r)>0,\, \forall \mu,\nu\in \M_+(X),\, \|\mu\|_{TV}, \|\nu\|_{TV} \leq r, \quad \left\|\Sigma[\mu] - \Sigma[\nu]\right\|_\infty \leq k(r)\|\mu - \nu\|_{TV}
\end{equation}
where $r\mapsto k(r)$ is a locally bounded function defined on $\R_+$. This assumption is reminiscent of the first assumption in Remark $2.2$ from~\cite{Desvillettes2008} in the context of $\Leb^1$ functions and~\cite{Canizo2012} in that of measures, but endowed with the dual bounded Lipschitz norm. Roughly, it can be interpreted as follows. Given two populations made of the same amount of people, the difference of their two selection pressures cannot grow too fast, compared to how different are these populations.
Then we assume
\begin{equation}\label{hyp:fitnessboundedfromabove}
\exists\,F>0\; \forall\, \mu \in\M_+(X),\quad \langle \mu,\Sigma[\mu]\rangle \leq F\mu(X).
\end{equation}
The quantity $\langle \mu,\Sigma[\mu]\rangle$ can interpreted as the fitness of a population described by a measure $\mu$, \textit{id est} the mean number of offsprings minus the mean number of deaths. Assumption~\eqref{hyp:fitnessboundedfromabove} thus enforces the population to grow at most of a factor $F$ per time unit. In addition, not imposing any lower bound to this value allows the population to eventually go extinct. This hypothesis is more general than the second one in Remark $2.2$ from~\cite{Desvillettes2008} in which the constant $F$ is replaced by a function $A_1 - A_2\left(\mu(X)\right)$ with $A_1 > 0$ and $A_2(z)\to\infty$ when $z\to\infty$. This stronger assumption though ensures that the total population remains bounded in large time, which is not the case with ours, see Remark~\ref{rmk:assumptions}. The ability for such a population to grow infinitely was already noted in~\cite{Canizo2012}, in which a slightly more general hypothesis than~\eqref{hyp:fitnessboundedfromabove}. Its equivalent in the langage deployed in the present paper would be
\[
\mu \mapsto \langle \mu,\Sigma[\mu]\rangle
\]
is bounded on every ball, which is straightforward with Hypothesis~\eqref{hyp:fitnessboundedfromabove}. These assumptions are enough to ensure well posedness results. The following one provides a sufficient condition for non extinction.
\beq \label{hyp:potentialgrowth}
\mu\in\M_+(X)~\text{and}~\Sigma~\text{are such that}~\mu(\left\{\Sigma[0] > 0\right\}) >0.
\eeq
It can be interpreted as the existence of a set of traits with positive $\mu-$measure that have the potential to proliferate, in the absence of competition.
\begin{theorem}\label{thm:wlpsdnss}
Assume the selection operator satisfies assumptions~\eqref{hyp:lips} and~\eqref{hyp:fitnessboundedfromabove}. Then for every nonnegative initial measure $\mu_0$, there exists a unique measure solution $(\mu_t)_{t\geq 0}$ to Equation~\eqref{eq:selection} in the sense of Definition~\ref{def:measure sol} that lies in $\Cr\left([0,T];\M_+(X)\right)$ for any $T>0$. In addition, if $(\nu_t)_{0\leq t\leq T}$ is a family of measures with nonnegative initial conditions $\nu_0$ that satisfies the same hypotheses, there exists a function $L=L(T)>0$ such that
\[
\forall\, t\in[0,T],\quad\|\mu_t - \nu_t\|_{TV} \leq \e^{L(T)t}\|\mu_0 - \nu_0\|_{TV}.
\]
In addition, for every $t\geq 0$, $\supp \mu_t \subset \supp \mu_0$.
If the additionnal hypothesis~\eqref{hyp:potentialgrowth} holds, then \[\inf_{t\geq0} \mu_t(X) > 0.\]
\end{theorem}

\begin{remark}\label{rmk:assumptions}
The assumptions above are enough to prove that such a measure solution is globally defined, but fail to ensure that the total population $\mu_t(X) = \|\mu_t\|_{TV}$ remains bounded for all times. Indeed, combining Assumption~\ref{hyp:fitnessboundedfromabove} with the mild formulation above, we obtain that for any $t\geq 0$, one has
\[
\mu_t(X) \leq \mu_0(X) + F\int_0^t \mu_s(X)\d s
\]
so Gr\"onwall's lemma provides
\[
\mu_t(X) \leq \mu_0(X)\e^{Ft}
\]
which ensures that the solution does not blow up in finite time. Now consider the very simple case $\Sigma[\mu] = \e^{-\|\mu\|_{TV}}\1_X$. One can show that this selection operator satisfies all the assumptions of the theorem
\[
\forall\, \mu \in\M(X),\; \Sigma[\mu] \in \mathcal{B}(X), \qquad \forall\, \mu \in\M_+(X),\; \langle\mu,\Sigma[\mu]\rangle = \|\mu\|_{TV}\e^{-\|\mu\|_{TV}} \leq \|\mu\|_{TV}
\]
(it is even uniformly bounded by $\e^{-1}$) and
\[
\forall\, \mu,\nu\in\M(X),\; \|\Sigma[\mu] - \Sigma[\nu]\|_\infty \leq \left|\|\mu\|_{TV} - \|\nu\|_{TV}\right| \leq \|\mu - \nu\|_{TV}
\]
but since $t\mapsto \|\mu_t\|_{TV}$ is solution of the ODE
\[
\ddt x = \e^{-x}x
\]
the total population goes to infinity in large time. An hypothesis in the flavour of `the planet is finite'~\cite{Smale1976} such as
\[
\exists M>0,\; \forall\, \mu\in\M(X),\quad \left[\|\mu\|_{TV} > M \Rightarrow \Sigma[\mu] \leq 0\right]
\]
ensures the boundedness of the population in large time. The phenomenon of infinite population in finite time is possible with a slightly weaker hypothesis than~\eqref{hyp:fitnessboundedfromabove}, see $(H5)$ in~\cite{Canizo2012}. In contrast, it is avoided in~\cite{Desvillettes2008} by such an assumption, that would translate here as the stronger condition
\[
\langle\mu,\Sigma[\mu]\rangle \leq \left(A_1 - A_2(\mu(X))\right)\mu(X)
\]
for all nonnegative finite measure $\mu$, with $A_1 >0$ and $A_2$ a function satisfyng $\lim_{z\to\infty} A_2(z) = \infty$.
\end{remark}

\section{Well-posedness and stability}

This section is devoted to the wellposedness result, stated on fairly general assumptions on the selection pressure operator. Our construction of a solution of Equation~\eqref{eq:measure} relies on a fixed-point method on the families of measures $\Cr([0,T];\M(X))$ with $T>0$ short enough, that we then iterate on time intervals of variable length. For a given $T>0$, the aforementionned space is a Banach space once endowed with the norm
\[
\sup_{t\in[0,T]} \|\mu_t\|_{TV}.
\]
Throughout the paper, when we refer to a family of measures, we used the notation $(\nu_t)_{0\leq t\leq T}$ or $(\nu)$. For such a family, we introduce a family of operators acting on the set of finite measures for $0\leq s\leq t\leq T$ by
\begin{equation}\label{eq:semigrp}
\mu M_{s,t}^{(\nu)} : f \mapsto \int_X f \e^{\int_s^t \Sigma[\nu_\sigma]\d \sigma} \d \mu.
\end{equation}
It is easy to see that this family defines a time inhomogeneous semigroup, since one can check that for $0\leq s\leq u\leq t\leq T$, it satisfies
\[
\left\lbrace
\begin{array}{l}
\mu M_{s,t}^{(\nu)} = \mu M_{s,u}^{(\nu)}M_{u,t}^{(\nu)} \medskip\\
\mu M_{s,s}^{(\nu)} = \mu.
\end{array}
\right.
\]
In order to prove the wellposedness of Equation~\eqref{eq:measure}, we first prove that there exists a unique family of measures denoted $(\mu_t)_{t\geq0}$ such that for all $t\geq0$, one has
\[
\mu_t = \mu_0 M_{0,t}^{(\mu)}
\]
or equivalently for all $t\geq0$ and $f\in\Cr_b(X)$
\[
\langle \mu_t,f\rangle = \int_X f \e^{\int_0^t \Sigma[\mu_s]\d s}\d \mu_0.
\]
As a first step, we prove the result for selection operators that are uniformly bounded from above.
%use a truncation: for $n\in \N$, let $\Sigma_n = \min(\Sigma,n)$. First, we prove we prove the wellposedness of Equation~\eqref{eq:measure} for any positive integer $n$.
\begin{lemma}\label{lem:fixdpnt}
Under the same assumptions as Theorem~\ref{thm:wlpsdnss}, if the selection pressure operator is in addition uniformly bounded from above, \textit{i.e.} 
\begin{equation}\label{hyp:boundedabove}
\exists\, n > 0,\; \forall\,\mu \in \M_+(X),\; \forall\, x\in X, \quad \Sigma[\mu](x) \leq n,
\end{equation}  
for any nonnegative initial data $\mu_0$ and final time $T>0$, there exists a unique family of measures $(\mu_t)_{0\leq t \leq T}\in \Cr([0,T];\M(X))$ such that for all $t\geq 0$
\beq \label{eq:fixpt}
\mu_t = \mu_0 M_{0,t}^{(\mu)}.
\eeq
Each measure of this family has the same support that of $\mu_0$.
\end{lemma}
\begin{proof}
For an initial datum $\mu_0\in\M(X)$, we want to prove that the function
\[
(\nu) \mapsto \mu_0 M_{0,t}^{(\nu)}
\]
has a unique fixed point. One can easily check that for any family $(\nu) \in \Cr\left([0,T];\M(X)\right)$ and final time $T>0$, the family $\left(\Gamma_\mu (\nu)\right)$ lies in $\Cr\left([0,T];\M(X)\right)$ and that if $\mu$ is nonnegative, then so is $\left(\Gamma_\mu(\nu)\right)_t$ for any $t\in [0,T]$. To start, we prove this lemma in short time, \textit{i.e.} for the final time $T>0$ small enough. For later prupose, let us introduce the set
\[
B_\mu^T := \left\{(\nu_t)_{0\leq t\leq T}\in\Cr([0,T];\M_+(X))|\, \nu_0 = \mu,\,\|\nu\|\leq 2\|\mu\|_{TV}\right\}.
\]
For $T$ a final time short enough so that if $(\nu)$ lies in $B_\mu^T$, so does $(\Gamma(\nu))$. Since $B_{\mu_0}^T$ is a closed subset of $\Cr([0,T];M(X))$, it is a complete metric space with the distance induced by the norm $\sup_{t\in[0,T]}\|\cdot\|_{TV}$, so we can apply the Banach fixed-point theorem.
\\
For $0\leq t \leq T$, $f\in\mathcal{B}(X)$ and $(\nu^1)$, $(\nu^2)$ two families of measures in $B_{\mu_0}^T$, we compute
\begin{align*}
\left|\left\langle\mu_0 M_{0,t}^{(\nu^1)} - \mu_0 M_{0,t}^{(\nu^2)}\right\rangle f\right| & \leq \|f\|_\infty \int_X \left|\e^{\int_0^t \Sigma[\nu^1_s]} - \e^{\int_0^t \Sigma[\nu^2_s]\d s}\right| \d \mu_0\\
	&\leq \|f\|_\infty\e^{nT}\int_X \int_0^t \left|\Sigma[\nu^1_s] - \Sigma[\nu^2_s]\right|\d s\, \d \mu_0
%	&\leq k(2\|\mu_0\|_{TV})\e^{nT}\|f\|_\infty\|\mu_0\|_{TV}\int_0^t \|\nu^1_s - \nu^2_s\|_{TV}\d s
\end{align*}
by the mean value inequality and the uniform boundedness from above hypothesis~\eqref{hyp:boundedabove}. Using assumption~\eqref{hyp:lips} and taking the supremum in $t$ over $[0,T]$, we obtain
\[
\sup_{t\in[0,T]}\left\|\mu_0 M_{0,t}^{(\nu^1)} - \mu_0 M_{0,t}^{(\nu^2)}\right\|_{TV} \leq k(2\|\mu_0\|_{TV})T\e^{nT}\|\mu_0\|_{TV}\sup_{t\in[0,T]}\|\nu^1_t - \nu^2_t\|_{TV}\]
with $k(2\|\mu_0\|_{TV})$ coming from the fact that $(\nu^1)$ and $(\nu^2)$ both lie in $B_{\mu_0}^T$. This means that $\Gamma$ is a contraction for a final time $T_1$ small enough, and thus admits a unique fixed point on $[0,T_1]$. We denote $(\mu_t)_{0\leq t\leq T_1}$ this family.
\\
Now, we extend the result to any finite time. A classical way to proceed would be to iterate the previous construction on successive time intervals $[T,2T]$, $[2T,3T]$... changing each time the initial datum by the final measure of the previous iteration. However, in the case under study, the contraction constant depends on the total variation norm of the initial measure, so the finite time is likely to change at each iteration. At each step, the fixed point theorem is applied in $B_{\mu_{T_j}}^{T_{j+1}}$, so for every integer $j$, the final time $T_j$ is smaller than $(\log 2)/n$ and the `sum of intermediate final times'
\begin{equation}\label{sumtimes}
\sum_{i=0}^N T_j
\end{equation}
does not trivially go to infinity as $N\to\infty$. Our goal now is to prove that under the assumptions of the lemma, this property is actually true. The mentionned iteration procedure gives, for each integer $j\geq1$
\begin{equation}\label{eq:infinitetime}
\|\mu_{T_{j+1}}\|_{TV} \leq \e^{nT_{j+1}}\|\mu_{T_j}\|_{TV} \leq \e^{n(T_1 + \cdots + T_{j+1})}\|\mu_0\|_{TV}.
\end{equation}
With the previous computations, for all $j\in \N$, in order to apply the fixed point theorem, the final time $T_{j+1}$ shall satisfy
\[
k(2\|\mu_{T_j}\|_{TV})T_{j+1}\e^{nT_{j+1}}\|\mu_{T_j}\|_{TV} < 1,
\]
that we rewrite as
\[
h(T_{j+1}) < \frac{1}{\|\mu_{T_j}\|_{TV}k(2\|\mu_{T_j}\|_{TV})}.
\]
Since $h$ is a non-negative strictly increasing function from $[0,\infty)$ to itself, we can define for all $j\in \N$ 
\[
x_j := h^{-1}\left(\frac{1}{\|\mu_{T_j}\|_{TV}k(2\|\mu_{T_j}\|_{TV})}\right) > 0.
\]
If the sequence $(x_j)_{j\in\N}$ does not converge to $0$, then we define $T_{j+1} := \frac{x_j}{2}$ and the series~\eqref{sumtimes} diverges. If the opposite is true, then we deduce from the definition of $x_j$ that $\|\mu_{T_j}\|_{TV}k(2\|\mu_{T_j}\|_{TV}) \to \infty$ as $j\to\infty$. Since the function $k$ is locally bounded, we deduce that this is actually true for $\|\mu_{T_j}\|_{TV}$, and finally using the estimate~\eqref{eq:infinitetime} the series~\eqref{sumtimes} diverges in this case too. In both cases, we can extend the family $(\mu)$ to any finite time.
\\
It is easy to see from~\eqref{eq:fixpt} that the support of $\mu_t$ is included in that of $\mu_0$. The converse inclusion is not not necessarily true, since it is not forbiden that the function $\Sigma[\mu]$ takes the value $-\infty$ inside the domain $X$.
\end{proof}
We are now ready to state the proof of Theorem~\ref{thm:wlpsdnss}. It relies on the truncation of the unbounded operator $\Sigma$ and on the previous lemma.
\begin{proof}{(Theorem~\ref{thm:wlpsdnss})}
Let $T>0$ and $\Sigma$ satisfying Assumptions~\eqref{hyp:lips} and~\eqref{hyp:fitnessboundedfromabove}. For all $n\in\N^*$, we define the truncation $\Sigma_n[\mu] = \min(\Sigma[\mu],n)$, \textit{id est}
\begin{equation*}
\Sigma_n[\mu](x) = \left\{
\begin{array}{l}
\Sigma[\mu](x)\text{ if }\Sigma[\mu](x) \leq n\medskip\\
n\text{ otherwise }
\end{array}
\right.
\end{equation*}
and denote $\Gamma_{\mu_0}^n$ the associated operator.
%show that this truncated operator satisfies the uniform boundedness from above of the previous lemma.
We also denote $(\mu^n_t)_{0\leq t\leq T}$ the unique fixed-point of this operator in $\Cr([0,T];\M(X))$, provided by the previous lemma. 
We show that this sequence remains bounded as $n\to\infty$.
Fix $t\in[0,T]$. We show that for any $n\in\N$, the family $(\mu^n)$ is a solution of the equation ~\eqref{eq:measure} with initial condition $\mu_0$ and $\Sigma_n$ instead of $\Sigma$. To this end, consider the function $\phi_n(t,x)$ defined by 
\[
\phi_n(t,x) = \left\{
\begin{array}{l}
\e^{\int_0^t \Sigma_n[\mu^n_\sigma](x)\d \sigma}\text{ if }x\in\supp \mu_0\medskip\\
0\text{ otherwise. }
\end{array}
\right.
\]
For almost all $x\in X$ and all $t\in [0,T]$ and $n\in\N$, the function $t\mapsto \phi_n(t,x)$ is differentiable and its derivative satisfies
\[
\left|\p_t\phi_n(t,x)\right| = \left|\Sigma[\mu_t^n](x)\e^{\int_0^t \Sigma[\mu_s^n](x)\d s}\right|\leq n \e^{n T}.
\]
By Leibniz integral rule, the function $t\mapsto \langle\mu^n_t,f\rangle$ is differentiable for every $f\in\mathcal{B}(X)$ and one has
\begin{equation}\label{eq:deriv_n}
\ddt \langle\mu^n_t,f\rangle = \langle\mu^n_t,\Sigma_n[\mu^n_t]f\rangle.
\end{equation}
Thanks to Equations~\eqref{eq:semigrp} and~\eqref{eq:fixpt}, we easily see that if $\mu_0$ is non negative, so is $\mu^n_t$ for all $t\in[0,T]$ and $n\in\N$. With this observation, the definition of $\Sigma_n$ and Assumption~\eqref{hyp:fitnessboundedfromabove} we obtain
\[
\ddt \|\mu_t^n\|_{TV} = \langle \mu^n_t,\Sigma[\mu^n_t]\rangle + \langle \mu^n_t,\underbrace{\Sigma_n[\mu^n_t]-\Sigma[\mu^n_t]}_{\leq 0}\rangle \leq F\|\mu_t^n\|_{TV}
\]
and finally thanks to Grönwall's lemma
\[
\|\mu_t^n\|_{TV} \leq \|\mu_0\|_{TV}\e^{Ft}
\]
so the sequence $\left(\|\mu_t^n\|_{TV}\right)_{n\in\N}$ is bounded for any $t> 0$.
\\
Now, we show that this sequence $(\mu^n_t)_{0\leq t\leq T}$ is actually constant from a certain rank, thus providing a family of measures $(\mu^\infty_t)_{0\leq t\leq T}$ that is its strong limit.
\\
The previous computations shows that for a fixed $t\geq 0$, the sequence of measures $(\mu_t^n)_{n\in\N}$ is uniformly bounded in total variation norm by $\|\mu_0\|_{TV}\e^{Ft}$. Thus we can bound the family of functions $\left(\Sigma[\mu^n_t]\right)_{n\in\N}$ uniformly on $\N$ by $\left(\|\mu_0\|_{TV}\e^{Ft}\right) k(\|\mu_0\|_{TV}\e^{Ft})+\left\|\Sigma[0]\right\|_\infty$. Indeed, applying Assumption~\eqref{hyp:lips} on the couple of measures $(\mu,0)$, we obtain
\[
\left\|\Sigma[\mu]-\Sigma[0]\right\|_\infty \leq k\left(\|\mu\|_{TV}\right)\|\mu\|_{TV}
\]
and then
\[
\left\|\Sigma[\mu]\right\|_\infty \leq k\left(\|\mu\|_{TV}\right)\|\mu\|_{TV} + \|\Sigma[0]\|_\infty.
\]
We deduce that
\[
\forall n \geq \left(\|\mu_0\|_{TV}\e^{Ft}\right) k(\|\mu_0\|_{TV}\e^{Ft}) + \|\Sigma[0]\|_\infty\; \forall\, s\in[0,t], \quad \Sigma_n[\mu^n_s] = \Sigma[\mu^n_s].
\]
So for $n$ large enough, familiar computations provide
\[\|\mu_t^{n+p} - \mu_t^n\|_{TV} \leq \e^{(\left(\|\mu_0\|_{TV}\e^{Ft}\right)k(\|\mu_0\|_{TV}\e^{Ft}) + \|\Sigma[0]\|_\infty)t}\|\mu_0\|_{TV}k(\|\mu_0\|_{TV}\e^{Ft}) \int_0^t \| \mu^{n+p}_s - \mu^n_s\|_{TV}\d s
\]
and finally Grönwall lemma provides the claimed result. Now we prove that the family of measures $(\mu_t^\infty)_{0\leq t\leq T}$ is a measure solution in the sense of Definition~\ref{def:measure sol}. Integrating Equation~\eqref{eq:deriv_n} in time with $n$ large enough, we obtain
\[
\langle \mu_t^n,f\rangle = \langle \mu_0,f\rangle + \int_0^t \langle \mu_s^n,\Sigma[\mu_s^n]f\rangle\ \d s.
\]
Since we proved that $\lim_{n\to\infty} \|\mu_t^n - \mu_t^\infty\|_{TV} = 0$, the left handside converges towards $\langle \mu_t^\infty ,f\rangle$. It remains to prove that
\[
\langle \mu_s^n,\Sigma[\mu_s^n]f\rangle \to \langle \mu_s^\infty,\Sigma[\mu_s^\infty]f\rangle
\]
for every $s\in[0,T]$ and $f\in\mathcal{B}(X)$. This is done by writing
\begin{align*}
\left| \langle \mu_s^n,\Sigma[\mu_s^n]f\rangle - \langle \mu_s^\infty,\Sigma[\mu_s^\infty]f\rangle \right| \leq & \left| \langle \mu_s^n - \mu_s^\infty,\Sigma[\mu_s^n]f\rangle \right| + \left|\langle \mu_s^\infty,\Sigma[\mu_s^n]f - \Sigma[\mu_s^\infty]f\rangle \right| \\
 \leq & \|f\|_\infty\left(2k\left(\|\mu_0\|_{TV}\e^{Ft}\right)\|\mu_0\|_{TV}\e^{Ft} + \|\Sigma[0]\|_\infty\right)\|\mu^n_s - \mu^\infty_s \|_{TV}
\end{align*}
thanks to some estimates previously established.
\\
Now we prove the stability result.
Let $\mu^1_0$ and $\mu^2_0$ two nonnegative measures on $X$ and denote $\mu^1_t$ and $\mu^2_t$ the corresponding solution at time $t\in[0,T]$, respectively, thus both satisfying Equation~\eqref{eq:measure}, for an arbitrary final time $T>0$. Then for $f\in\mathcal{B}(X)$, one has
\begin{align*}
|\mu^1_tf - \mu^2_tf| &\leq |\mu^1_0f - \mu_0^2f| + \int_0^t\langle\mu_s^1,\left|\left(\Sigma[\mu_s^1] - \Sigma[\mu_s^2]\right)f\right|\rangle + \left|\langle\mu_s^1 - \mu_s^2,\Sigma[\mu_s^2]f\rangle\right|\d s\\
	&\leq \|f\|_\infty\left(\|\mu_0^1 - \mu_0^2\|_{TV} + L(T)\int_0^t \|\mu^1_s - \mu^2_s\|_{TV}\right)
\end{align*}
with
\[
L(T) := 2\sup_{t\in[0,T]}\left((\|\mu_0^1\|_{TV} + \|\mu_0^2\|_{TV})\e^{Ft}k((\|\mu_0^1\|_{TV} + \|\mu_0^2\|_{TV})\e^{Ft}) + \|\Sigma[0]\|_\infty\right)
\]
and finally by Grönwall's lemma provide the claimed inequality. This stability result provides the uniqueness of the solution. Finally, the claim on the support of $\mu_t$ comes from the one from Lemma~\ref{lem:fixdpnt}, but supplemented with the boundedness in supremum norm of $\Sigma[\mu]$, that fordids extinction of a trait in finite time. Thus, if $\supp \mu_0 \neq \emptyset$, then $\|\mu_t\|_{TV} >0$ for any $t>0$. The proof of the final claim in large time is postponed to the next section.
\end{proof}

\section{Various asymptotic behaviours}

In this section, we give a sufficient condition of non extinction and provide some examples of selection operators to illustrate different dynamics encompassed by our assumptions.

\subsection{A sufficient condition for non extinction}

Our rather general assumptions allows various dynamics to occur. For example, a system studied in~\cite{May1975} might be rewritten $X = {x_0,x_1,x_2}$, and
\begin{equation}\label{Discrete_osc}
\Sigma[\mu](x_i) = 1 - \mu(\{x_i\}) - 2\mu(\{x_{i+1}\})
\end{equation}
with indices in $\Z/3\Z$
and we easily check that such function satisfies hypotheses~\eqref{hyp:lips} and~\eqref{hyp:fitnessboundedfromabove} with
\[
\|\Sigma[\mu^1] - \Sigma[\mu^2]\|_\infty \leq 2 \|\mu^1 - \mu^2\|_{TV}\qquad \text{and} \qquad \langle\mu,\Sigma[\mu]\rangle = |\mu\|_{TV}\left(1-|\mu\|_{TV}\right)\leq |\mu\|_{TV}.
\]
This system displays a periodic behaviour, as proved in the original paper. Otherwise, the solution to equation~\eqref{eq:measure} can either go extinct, converge to an equilibrium or display a chaotic behaviour. Since this last option is not really considered, we have to decide between the two first possibilities. This is the purpose of the following lemma.
\begin{lemma}
Assume the selection operator satisfies hypotheses~\eqref{hyp:lips} and~\eqref{hyp:fitnessboundedfromabove}. Let $\mu_0$ be an initial condition satisfying Assumption~\eqref{hyp:potentialgrowth}, and denote $(\mu_t)_{t\geq 0}$ the unique measure solution to~\eqref{eq:selection} with initial condition $\mu_0$. If \[\mu_t \xrightarrow[t\to\infty]{TV} \mu,\] then $\mu\neq 0$.
\end{lemma}
\begin{proof}
Assume by contradiction that $\mu = 0$. Then for all $\epsilon >0$, there exists $T>0$ such that for all $t\geq T$, one has $\|\mu_t\|_{TV} < \epsilon$. Fix $\eta > 0$ as small as needed. Thanks to Hypothesis~\eqref{hyp:potentialgrowth}, there exists $A \subset \left\{\Sigma[0] > 0\right\}$ with $\mu_0(A)>0$ and for all $x\in A$, $\Sigma[0](x) > \eta$. With Hypotheses~\eqref{hyp:lips}, one has for all $t\geq T$
\[
\left\|\Sigma[\mu_t] - \Sigma[0]\right\|_\infty \leq \epsilon k(\epsilon),
\]
thus, for all $x\in A$, one has the estimate
\[
\Sigma[\mu_t](x) > \eta - \epsilon k(\epsilon) > \frac{\eta}{2}
\]
for $\epsilon$ small enough, \textit{i.e.} $T$ large enough. This provides a contradiction once it is noted that
\[\epsilon > \|\mu_t\|_{TV} \geq \int_A \e^{\int_0^t \Sigma[\mu_s] \d s} \d \mu_0\ \geq C_T\int_A \e^{\int_T^t \Sigma[\mu_s] \d s} \d \mu_0 > C_T\e^{\frac{\eta}{2}(t - T)}\mu_0(A)\]
with $C_T$ a positive constant, for all $t\geq T$.
\end{proof}

\subsection{Cannibalism revisited\label{sec:canni}}

In this section, we borrow an example from~\cite{PerthameTransport}. In this example, the trait $x$ lies in $X = [0,\infty)$ and represent the degree of cannibalism. For a measure $\mu$ that have both finite zeroth and first moment, the selection operator is
\begin{equation}\label{op:canni}
\Sigma[\mu](x) = r + \alpha x \mu(X) - \langle \mu,Id\rangle
\end{equation}
with $r$ the growth rate in the absence of cannibalism and $\alpha \in (0,1]$ the efficiency in offspring production from intraspeciﬁc predation. As noted in this book, unbounded levels of predation seems very unrealistic, so one can consider the trait instead in a compact set $[0,A]$ with $A>0$. This hypothesis is even necessary in the present paper, since otherwise the operator $\Sigma[\mu]$ does not lie in $\Leb^\infty$. In this setting, we recover the result presented in the book.
\begin{proposition}
There is a unique global positive measure solution $(\mu_t)_{t\geq 0}$ to the equation given by the selection operator~\eqref{op:canni}. In addition, if $M:=\sup \supp \mu_0 \in \supp \mu_0$, and $\langle \mu_0,Id\rangle \leq \frac{r}{1-\alpha}$, one has the asymptotic concentration on the trait $M$, \textit{i.e.}
\[
\lim_{t\to\infty} \mu_t = \frac{r}{M(1-\alpha)} \delta_M.
\]
\end{proposition}
\begin{proof}
First, we easily see that the operator given by~\eqref{op:canni} satisfies the assumptions of Theorem~\ref{thm:wlpsdnss}, since $X$ is bounded. Indeed, one has
\[
\forall\, \mu\geq 0,\;\langle \mu,\Sigma[\mu]\rangle = r\mu(X) -(1-\alpha)\mu(X)\langle \mu,Id\rangle \leq r\mu(X)
\]
and
\[
\forall\, \mu,\nu,\; \left\|\Sigma[\mu] - \Sigma[\nu]\right\|_\infty \leq (1+\alpha)A\|\mu-\nu\|_{TV}.
\]
Without loss of generality, we choose a family of measures solutions $(\mu_t)_{t\geq 0}$ such that $\mu_0(X) = 1$. Throughout the proof, we will use a rescaled family of measures defined by
\begin{equation}\label{eq:munu}
\nu_t := \e^{-\int_0^t \left(r - (1-\alpha)\langle \mu_s,Id\rangle\right)\d s}\mu_t
\end{equation}
and easily check that it satisfies
\[
\langle \nu_t,f\rangle = \frac{\langle \mu_t,f\rangle}{\mu_t(X)}
\]
so for all $t\geq 0$, $\nu_t$ is a probability measure on $X$. In addition, for all $t\geq0$ and $f\in\mathcal{B}(X)$ one has
\[
\langle \nu_t,f\rangle = \langle \mu_0,f\rangle + \int_0^t \langle \nu_s, \alpha\left(\mu_s(X)Id - \langle \mu_t,Id\rangle\right)f\rangle.
\]
We have the inequality
\[
\mu_t(X)\langle \mu_t,Id^2\rangle = \mu_t(X)^2\langle \nu_t,Id^2\rangle \geq \mu_t(X)^2\langle \nu_t,Id\rangle^2 = \langle \mu_t,Id\rangle^2
\]
using Jensen's inequality. Now using this inequality, we compute
\begin{align}
\ddt \langle \mu_t,Id\rangle &= \left(r \langle \mu_t,Id\rangle + \alpha \mu_t(X)\langle \mu_t,Id^2\rangle - \langle \mu_t,Id\rangle^2\right)\nonumber\\
& \geq \langle \mu_t,Id\rangle\left(r - (1-\alpha)\langle \mu_t,Id\rangle\right).\label{ineq:muId}
\end{align}
Assume by contradiction that
\[
\lim_{t\to\infty} \langle\mu_t,Id\rangle = 0.
\]
Fix $\epsilon >0$. There exists $t_0>0$ such that for all $t\geq t_0$, one has $\langle \mu_t,Id\rangle \leq \epsilon$. Integrating~\eqref{ineq:muId} for $t\geq t_0$, one has
\[
\langle \mu_t,Id\rangle \geq \langle \mu_0,Id\rangle C \e^{\epsilon (t-t_0)}
\]
with $C$ a positive constant, yielding a contradiction. Now we compute
\begin{align*}
\ddt \left(r - (1-\alpha)\langle \mu_t,Id\rangle\right) &= -(1-\alpha)\left(r \langle \mu_t,Id\rangle + \alpha \mu_t(X)\langle \mu_t,Id^2\rangle - \langle \mu_t,Id\rangle^2\right)\\
& \leq -(1-\alpha)\langle \mu_t,Id\rangle\left(r - (1-\alpha)\langle \mu_t,Id\rangle\right).
\end{align*}
and Gr\"onwall's lemma yields
\begin{equation}\label{ineq:gronwall}
r - (1-\alpha)\langle \mu_t,Id\rangle \leq (r - (1-\alpha)\langle \mu_0,Id\rangle) \e^{-(1-\alpha)\int_0^t \langle \mu_s,Id\rangle \d s}.
\end{equation}
In addition, since $\langle \mu_t,Id\rangle$ does not vanish, one has
\[
\e^{-(1-\alpha)\int_0^t \langle \mu_s,Id\rangle \d s} \leq \e^{-(1-\alpha)\epsilon t}
\]
for some $\epsilon >0$. Since by assumption $\langle\mu_0,Id\rangle \leq \frac{r}{1-\alpha}$, the last inequality combined with~\eqref{ineq:gronwall} provides
\[
\lim_{t\to\infty} r - (1-\alpha)\langle \mu_t,Id\rangle \leq 0.
\]
Assume by contradiction that this limit is negative. Since
$t\mapsto r - (1-\alpha)\langle \mu_t,Id\rangle$ is continuous, there exists $\epsilon>0$ and $t_0>$ such that for all $t\geq t_0$, one has
\[
r - (1-\alpha)\langle \mu_t,Id\rangle \leq -\epsilon < 0.
\]
 For $t\geq t_0$, one has
\[
\mu_t(X) =  \e^{\int_0^t \left(r - (1-\alpha)\langle \mu_s,Id\rangle\right)\d s} \leq C\e^{-\epsilon (t-t_0)}
\]
with $C$ a positive constant. Thus, for such $t$, one has
\[
\langle \mu_t,Id\rangle \leq M\mu_t(X) \leq MC\e^{-\epsilon (t-t_0)}
\]
and finally
\[
r - (1-\alpha)\langle \mu_t,Id\rangle \geq r - (1-\alpha)MC\e^{-\epsilon t} > 0
\]
for $t$ large enough, yielding a contradiction. Finally, we obtain
\begin{equation}\label{lim:mu}
\lim_{t\to\infty} \langle \mu_t,Id\rangle = \frac{r}{1-\alpha},
\end{equation}
from which we also deduce that $t\mapsto \mu_t(X)$ does not vanish. Now we prove that the variance of the measures $(\nu_t)_{t\geq0}$ vanishes. To obtain this property, we write
\begin{align*}
\langle \nu_t,Id\rangle &= \langle \mu_0,Id\rangle + \int_0^t \langle \nu_s, \alpha\left(\mu_s(X)Id^2 - \langle \mu_t,Id\rangle Id\right)\rangle \\
& = \langle \mu_0,Id\rangle + \alpha\int_0^t \mu_s(X)\langle \nu_s,Id^2\rangle - \langle \mu_s,Id\rangle \langle \nu_s,Id\rangle \d s \\
& = \langle \mu_0,Id\rangle + \alpha\int_0^t \mu_s(X)\left[\langle \nu_s,Id^2\rangle - \langle \nu_s,Id\rangle^2 \right]\d s
\end{align*}
and Jensen's inequality again ensures that the integrand is non negative. Since every other term is also non negative and $\langle \nu_t,Id\rangle \leq M$, the integral on the right handside is finite, and thus the variance of $(\nu_t)_{t\geq 0}$ vanishes. Since $\nu_t$ does not vanish, it means that this family of measures concentrates on some point $x^*\in [0,A]$, and so does $\mu_t$. Now assume by contradiction that $x^* < M$. Then one can find a positive number $\eta < M-x^*$. We compute
\begin{align*}
\ddt \nu_t([M-\eta,M]) &= \alpha \langle \nu_t,\left(Id \mu_t(X) - \langle \mu_t,Id\rangle\right)\1_{[M-\eta,M]}\rangle \\
&\geq \alpha \mu_t(X)\left((M-\eta - \langle \nu_t,Id\rangle\right)\nu_t([M-\eta,M])\\
&\geq \alpha m\left((M-\eta - x^*\right)\nu_t([M-\eta,M])
\end{align*}
since $t\mapsto \langle \nu_t,Id\rangle$ is increasing, with $0<m\leq \mu_t(X)$. This provides
\[
\nu_t([M-\eta,M]) \geq \nu_0([M-\eta,M])\e^{\alpha m (M-\eta - x^*)t}
\]
which is a contradiction, so one has
\begin{equation}\label{lim:nu}
\nu_t \rightharpoonup_{t\to\infty} \delta_M.
\end{equation}
Finally, we combine~\eqref{eq:munu}~\eqref{lim:mu} and~\eqref{lim:nu}, to obtain that
\[
\lim_{t\to\infty} \mu_t(X) = \lim_{t\to\infty} \e^{\int_0^t \left(r - (1-\alpha)\langle \mu_s,Id\rangle\right)\d s} = \frac{r}{M(1-\alpha)}
\]
which ends the proof.
\end{proof}
The properties highlighted in the previous proposition are illustrated in Figure~\ref{fig:canni}, with parameters $r=3$, $\alpha = 0.8$ and $M=1$.
\begin{figure}
\includegraphics[width=\textwidth]{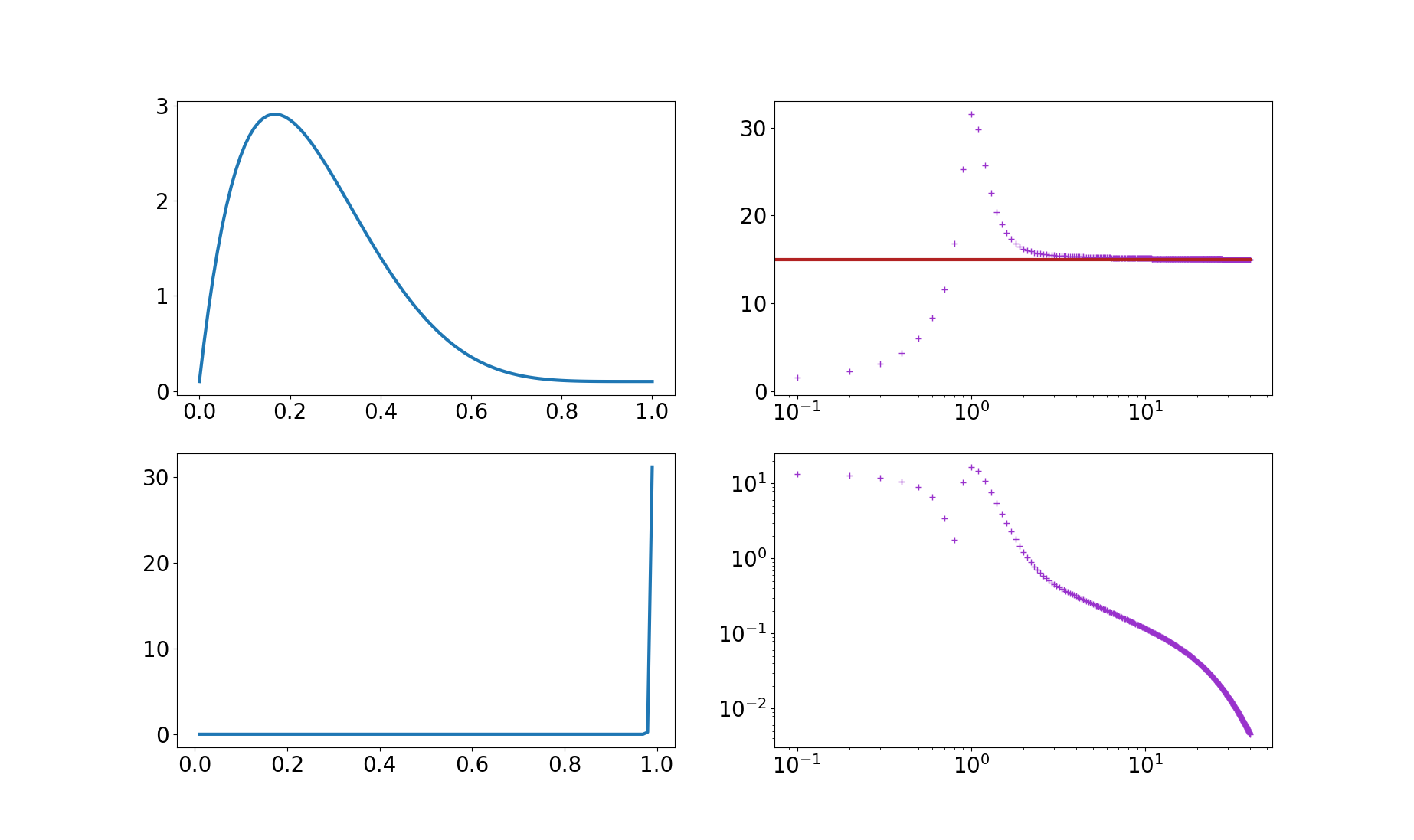}
\caption{Top left: the initial distribution is given by the sum of a Beta distribution with parameters $2$ and $6$ and a constant, namely $x\mapsto \beta_{(2,6)}(x) + 0.1$. Bottom left: the distribution at time $T=40$. Top right: purple: total population at the discretization times, red: the function $t\mapsto r/(M(1-\alpha)) = 15$. Bottom right: $\Leb^1$ difference between the total population at time $t$ and the $r/(M(1-\alpha)) = 15$.\label{fig:canni}}
\end{figure}
With notations closer to the ones used in~\cite{Desvillettes2008}, the selection operator would be
\[
\Sigma[\mu](x) = r - \int_X (y-\alpha x)\d \mu(y)
\]
which does not satisfy the assumptions required for existence in their paper, but does satisfy the one from the present paper. As stated earlier, the total variation norm is not well suited for cases in which concentration occurs. The purpose of the three next subsections is to provide example of such favorable cases for this norm.
\subsection{A stable distribution without singular part}

In this section, we consider a selection operator given by
\[
\Sigma[\mu](x) = a(x) - \int_X J(x-y)\d \mu(y)
\]
defined on $X = \R$, with
\[
a(x) = \frac{(1+|x|)\e^{-|x|}}{4} \qquad \text{and the kernel} \qquad J(x) = \frac{\e^{-|x|}}{2}.
\]
This operator $\Sigma$ satisfies the assumptions of the existence section. Indeed, one has
\[
\forall\, \mu\geq 0,\;\langle \mu,\Sigma[\mu]\rangle = \langle\mu,a\rangle - \int_X J(x-y)\d \mu(y)\d \mu(x) \leq \mu(X)
\]
and
\[
\forall\, \mu,\nu,\; \left\|\Sigma[\mu] - \Sigma[\nu]\right\|_\infty \leq \frac{1}{2}\|\mu-\nu\|_{TV}.
\]
It has to be noted that $a = J*J$ with $*$ being the convolution operator. In the $\Leb^1$ context, a steady state $u$ would satisfy \[\left(a(x) - J*u(x)\right)u(x)=0\] almost everywhere, and we are interested in positive solutions. Applying the Fourrier transform, such positive steady state would satisfy \[\widehat{J}(\xi)\left(\widehat{J}(\xi) - \widehat{u}(\xi)\right) = 0\] so a natural candidate for a steady state is the kernel $J$ itself.
\\
We perform numerical tests on a truncated version of the problem, namely
\[
X = [-h,h], \qquad J_h(x) = \frac{\e^{-|x|}}{2(1-\e^{-h})}\1_{(-h,h)} \qquad \text{and} \qquad a_h(x) = \frac{(1+|x|)\e^{-|x|} - \e^{-2h}\cosh(x)}{4(1-\e^{-h})^2}.
\]
For various initial conditions, the distribution seems to converge towards $J_h$, see Figure~\ref{fig:J} for an example. In this subsection and the next, the numerical method used is a standard semi-implicit Euler scheme.
\begin{figure}
\includegraphics[width=\textwidth]{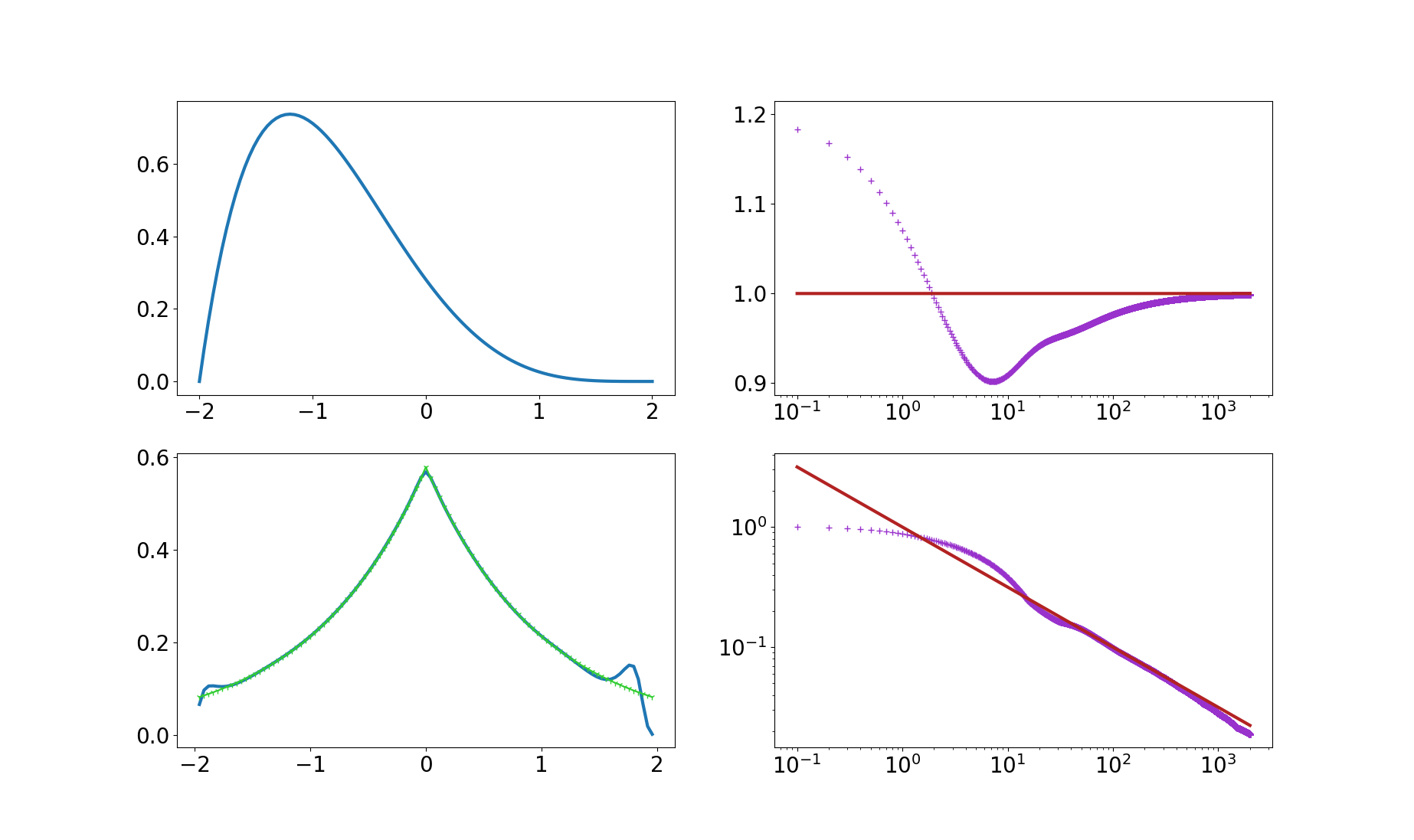}
\caption{Top left: the initial distribution is given by $x\mapsto \beta_{(2,5)}\left(\frac{x+h}{2h}\right)$, with $\beta_{(2,5)}$ a Beta distribution with parameters $2$ and $5$. Bottom left: blue: the distribution at time $T=2000$, green: the function $J_h$. Top right: purple: total population at the discretization times, red: the function $t\mapsto 1$. Bottom right: purple: $\Leb^1$ difference between the solution and the fonction $J_h$. Red: the function $t\mapsto t^{-\frac{1}{2}}$.\label{fig:J}}
\end{figure}
The function $J_h$ seems to be a stable steady state. In a measure setting, we can expect a convergence in total variation norm, but rather slowly.

\subsection{Trait-structured preys-predators}

In the previous subsection, we provided an \textit{ad hoc} example of operator such that the solutions converges towards an equilibrium without singular part which is stable with respect to the initial condition. In this one, we give an example which is more biologically grounded, that also seems to converge towards a measure that has a density with respect to the Lebesgue measure. To this end, we consider the trait space $X = [0,1]$ and a selection operator given by
\[
\Sigma[\mu](x) = a(x) + A\mu([x-\eta,x)\cap X) - B\mu([x,x+\eta)\cap X)
\]
with positive constants $A$ and $B$. This function intends to model a preys-predators type interaction with trait $x$ being interpreted as the position in the food chain. Each individual can be both prey for and predator depending on the value of its trait $x$. In addition, the function $a$ is taken decreasing, to model the ability for smaller species to proliferate faster. More precisely, in the simulation, we take $a(x) = 1 - 1.5\sqrt{x}$, $A = 0.8$, $B=0.7$ and $\eta = 0.51$. In this setting, the selection operator satisfies Assumption~\eqref{hyp:lips} with $k(r) = A+B$, but it is unclear if it also satisfies Assumption~\eqref{hyp:fitnessboundedfromabove}. We performed the simulations anyway and obtained Figure~\ref{fig:pp}. We note that with these parameters, the asymptotic measure seems to have no singular part. In addition, we note first that with the predation phenomenon, a population with nonnegative proliferation rate can survive, and second that dumped oscillation in the total population occurs, which is reminiscent of the classical preys-predators system.
\begin{figure}
\includegraphics[width=\textwidth]{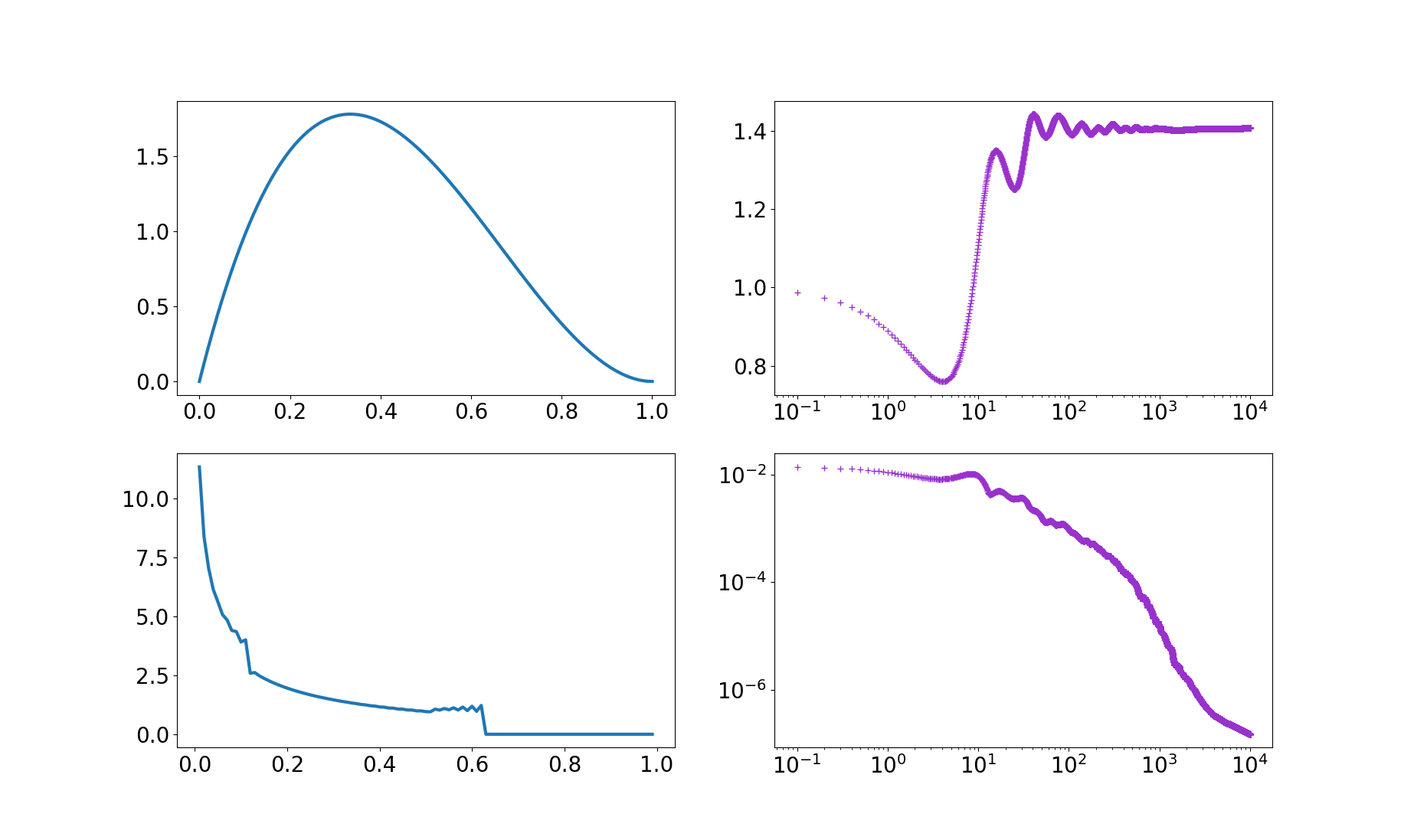}
\caption{Top left: the initial distribution is a Beta distribution with parameters $2$ and $3$ $x\mapsto \beta_{(2,3)}(x)$. Bottom left: the distribution at time $T=10000$. Top right: total population at the discretization times. Bottom right: $\Leb^1$ difference between two successive discretized solution.\label{fig:pp}}
\end{figure}
In the notation of~\cite{Desvillettes2008}, the selection operator would be
\[
\Sigma[\mu](x) = a(x) - \int_0^1 B\1_{[x,x+\eta)\cap [0,1]}(y) - A\1_{[x-\eta,x)\cap [0,1]}(y)
\]
which does not satisfies their assumptions for existence.
\subsection{Convergence in total variation norm with a uniform competition for ressources}

In this subsection, we consider a famous particular model for the selection equation, given by the selection operator
\begin{equation}\label{op:camille}
\Sigma[\mu](x) = r(x) - \mu(X)
\end{equation}
with $X$ an arbitrary compact set, say $[0,1]$. This equation has been studied in~\cite{PerthameTransport} when $r$ has a single maximum, and in~\cite{Lorenzi_2020} when different species can coexist asymptotically, see also~\cite{Pouchol_2018} for global stability. Here, we consider plateau growth rates $r$, in the sense that the maximum is not reached on a discrete set. More precisely, we require
\begin{equation}\label{hyp:plateau}
r:X\to (0,\infty), \qquad \max_{x\in X} r(x) =: r_M, \qquad \min_{x\in X} r(x) =: r_m > 0
\end{equation}
and
\begin{equation}\label{hyp:gap}
\inf_{y\not\in \mathcal{S}} r_M - r(y) =:\eta >0
\end{equation}
in which we denote
\[
\mathcal{S}:= \mathrm{argmax}(r).
\]
We require the initial measure $\mu$ to be finite and that a subset of $\mathcal{S}$ is included in its support 
\begin{equation}\label{hyp:nonempty}
\mathcal{S}_\mu := \supp \mu \cap \mathcal{S} \neq \emptyset.
\end{equation}
We also need the hypothesis
\begin{equation}\label{hyp:ac}
\1_\mathcal{S}\mu : f\mapsto \langle \mu,f\1_\mathcal{S}\rangle~\text{is absolutely continuous w.r.t. the Lebesgue measure}~\mathcal{L}
\end{equation}
and finally
\begin{equation}\label{hyp:lebesgue}
\mathcal{L}(\mathcal{S}_\mu) > 0.
\end{equation}
Under these assumptions, we do not observe the usual concentration on a discrete set, but the dynamics rather selects the traits in $\mathcal{S}_\mu$, as stated in the following proposition.
\begin{Proposition}\label{prop}
Under Assumptions~\eqref{hyp:plateau},~\eqref{hyp:gap},~\eqref{hyp:nonempty}
,~\eqref{hyp:ac} and~\eqref{hyp:lebesgue}, the measure solution $(\mu_t)_{t\geq0}$ associated with the operator~\eqref{op:camille} satisfies
\begin{align*}
&\left\|\mu_t - \e^{\int_0^\infty r_M-\mu_s(X)\d s}\1_{S}\mu\right\|_{TV} \leq C_1(\mu,r)\e^{-\frac{\eta}{2}t}\|\mu - \1_{S}\mu\|_{TV}\\
&\qquad\qquad+\leq C_2(\mu,r)\e^{-\eta t} + C_3(\mu,r)\e^{-r_M t}\|r - \mu_0(X)\|_\infty
\end{align*}
with $C_1(\mu,r)$, $C_2(\mu,r)$ and $C_3(\mu,r)$ constants depending on the initial condition and the growth rate $r$.
\end{Proposition}
The result is achieved by combining the two following lemmas. The first one states that the dynamics selects the traits that lies in $\mathcal{S}_\mu$, and the second deals with the asymptotic behaviour of the total population.
\begin{lemma}
Under the same assumptions as in Proposition~\ref{prop}, one has
\[
\left\|\mu_t - \e^{\int_0^t r_M-\mu_s(X)\d s}\1_{S}\mu\right\|_{TV} \leq C_1(\mu,r)\e^{-\frac{\eta}{2}t}\|\mu - \1_{S}\mu\|_{TV}.
\]
\end{lemma}
\begin{proof}
We first consider the dual problem, namely a classical solution $f$ of
\[
\p_t f(t,x) = (r(x) - \mu_t(X))f(t,x)
\]
with initial condition $f_0$. A solution can be writen
\[
f(t,x) = \e^{\int_0^t \left(r(x) - \mu_s(X)\right)\d s}f_0(x).
\]
Let us consider $g$ be the classical solution of the related problem
\[
\p_t g(t,x) = (r(x) - r_M)g(t,x)
\]
with initial condition $f_0$. We can express $g$ as
\[
g(t,x) = \e^{(r(x) - r_M)t}f_0(x)
\]
and its relation to $f$ by
\[
f(t,x) = \e^{\int_0^t \left(r_M - \mu_s(X)\right)\d s}g(t,x).
\]
Thanks to Assumptions~\eqref{hyp:plateau} and~\eqref{hyp:gap}, one has for all $t\geq0$ and $x\in X$
\[
\left|g(t,x) - \1_{\mathcal{S}}(x)f_0(x)\right| \leq \e^{-\eta t}|f_0(x) - \1_{\mathcal{S}}(x)f_0(x)|
\]
so we obtain
\[
\left|f(t,x) - \e^{\int_0^t \left(r_M - \mu_s(X)\right)\d s}\1_{\mathcal{S}}(x)f_0(x)\right| \leq \e^{\int_0^t \left(r_M - \mu_s(X)-\eta\right)\d s}|f_0(x) - \1_{\mathcal{S}}(x)f_0(x)|.
\]
It is proved in~\cite{Lorenzi_2020} that $\mu_t(X) \to r_M$, so there exists $t_0\geq 0$ such that for all $t\geq t_0$, one has
\[
|r_M - \mu_t(X)|\leq \frac{\eta}{2},
\]
from which we deduce
\[
\left|f(t,x) - \e^{\int_0^t \left(r_M - \mu_s(X)\right)\d s}\1_{\mathcal{S}}(x)f_0(x)\right| \leq \e^{\left(r_M - \frac{\eta}{2}\right)t_0}\e^{-\frac{\eta}{2}t}|f_0(x) - \1_{\mathcal{S}}(x)f_0(x)|.
\]
Using Assumptions~\eqref{hyp:ac} and~\eqref{hyp:lebesgue}, we can take the dual inequality in total variation norm, which ends the proof.
\end{proof}
Now we provide a control of the difference between
\[
t\mapsto e^{\int_0^t r_M-\mu_s(X)\d s}
\]
and its final value.
\begin{lemma}
Under Assumptions~\eqref{hyp:plateau} and~\eqref{hyp:gap}, the function $t\mapsto r_M - \mu_t(X)$ lies in $\Leb^1(0,\infty)$ and for all $t\geq 0$, one has
\[
\left|e^{\int_0^t r_M-\mu_s(X)\d s} - e^{\int_0^\infty r_M-\mu_s(X)\d s}\right| \leq \frac{\eta}{r_m\min \left(1,\frac{\mu_0(\mathcal{S})}{r_M}\right)}\e^{-\eta t} + \frac{R_M}{r_m \min \left(1,\frac{\mu_0(\mathcal{S})}{r_M}\right)}\e^{-r_M t}\|r - \mu_0(X)\|_\infty.
\]
\end{lemma}
\begin{proof}
Drawing inspiration from~\cite{Lorenzi_2020}, one has
\[
\mu_t(X) = \frac{\left\langle \mu_0, \e^{r(\cdot)t}\right\rangle}{1 + \left\langle \mu_0, \frac{\e^{r(\cdot)t}-1}{r}\right\rangle}
\]
so
\[
r_M - \mu_t(X) = \frac{\left\langle \mu_0, \left(\frac{r_M}{r}-1\right)\e^{r(\cdot)t}\right\rangle}{1 + \left\langle \mu_0, \frac{\e^{r(\cdot)t}-1}{r}\right\rangle} + \frac{r_M - \left\langle \mu_0, \frac{r_M}{r}\right\rangle}{1 + \left\langle \mu_0, \frac{\e^{r(\cdot)t}-1}{r}\right\rangle}.
\]
To estimate these terms, we first notice that
\[
1 + \left\langle \mu_0, \frac{\e^{r(\cdot)t}-1}{r}\right\rangle \geq 1 + \frac{\e^{r_M t} - 1}{r_M}\mu_0(\mathcal{S}) \geq \min \left(1,\frac{\mu_0(\mathcal{S})}{r_M}\right)\e^{r_M t}.
\]
Then, one has
\begin{align*}
\left\langle \mu_0, \left(\frac{r_M}{r}-1\right)\e^{r(\cdot)t}\right\rangle &= \int_{\mathcal{S}^c \,\cap\, \supp \mu_0}\left(\frac{r_M}{r(y)}-1\right)\e^{r(y)t}\d \mu_0(y)\\
\leq \frac{\eta}{r_m}\e^{(r_M-\eta)t}
\end{align*}
so the first term, which is non negative, is controlled by
\[
\frac{\left\langle \mu_0, \left(\frac{r_M}{r}-1\right)\e^{r(\cdot)t}\right\rangle}{1 + \left\langle \mu_0, \frac{\e^{r(\cdot)t}-1}{r}\right\rangle} \leq \frac{\eta}{r_m\min \left(1,\frac{\mu_0(\mathcal{S})}{r_M}\right)}\e^{-\eta t}.
\]
For the second, we write
\[
\left|r_M - \left\langle \mu_0, \frac{r_M}{r}\right\rangle\right| = \left|\left\langle \mu_0,r_M\left(\frac{1}{\mu_0(X)} - \frac{1}{r}\right)\right\rangle \right|\leq \frac{R_M}{r_m}\|r - \mu_0(X)\|_\infty
\]
so finally
\[
\left|\frac{r_M - \left\langle \mu_0, \frac{r_M}{r}\right\rangle}{1 + \left\langle \mu_0, \frac{\e^{r(\cdot)t}-1}{r}\right\rangle}\right| \leq \frac{R_M}{r_m \min \left(1,\frac{\mu_0(\mathcal{S})}{r_M}\right)}\e^{-r_M t}\|r - \mu_0(X)\|_\infty
\]
and the proof is complete.
\end{proof}

\section{Discussion}

In this work, we have studied the classic pure selection equation in the framework of measures. It enables to obtain well posedness of a global solution for fairly general assumptions, as well as a sufficient hypothesis for the persistence of the population that is readily interpreted. Then, we explored various classes of selection operator, both theoretically and numerically, and obtained different kinds of behaviours.
\\
All our theoretical study took place in the context of the topology of the total variation. Although this norm might seem 'rigid' for models stemming from adaptative dynamics, in which convergence towards Dirac deltas can occur, the examples we studied highlighted that under some particular assumptions, convergence in total variation norm is possible, even exponentially fast. Such decay estimates are, up to our knowledge, new for the selection equation. One possible continuation of this work would be to obtain such decay estimates in bounded Lipschitz norm, more suited for cases in which concentration happens. In particular, the case of subsection~\ref{sec:canni} display numericall a very fast convergence towards a Dirac mass, so it would be no surprise if an exponential decay could be proved.
\\
The examples we presented displayed various behaviours. In their paper mentionned earlier~\cite{May1975}, the authors highlighted sustained oscillation in a simple ODE system. One might wonder if it is possible to exhibit a continous version of their model, in which their would be a continuum of traits instead of three separate ones. Another way to obtain oscillations would be to include a periodic term in the selection operator, in the fashion of~\cite{Carrere_2020}.

\section*{Ackowledgement}

The author is grateful for the comments of the two anonymous reviewers, which greatly helped to improve the first manuscript. The author also thanks François \textsc{Castella} for a critical reading of the second manuscript. This work has been partially supported by the Chair “Modélisation Mathématique et Biodiversité” of Veolia Environnement-Ecole Polytechnique-Museum National d’Histoire Naturelle-Fondation X.

\section*{Supplementary mterials}
The codes used to generate the figures are available at \url{https://github.com/Hugo-Martin/selection_equation.git}.

\bibliographystyle{plain}

%This calls all references from the .bib
%\nocite{*}

%  This inserts the bib file
\bibliography{biblio}

\end{document}